\documentclass[12pt,reqno]{amsart}
\usepackage{amsmath, amsthm, amssymb}
\makeatletter
\@namedef{subjclassname@2010}{%
  \textup{2010} Mathematics Subject Classification}
\makeatother

\usepackage{color}
\usepackage{comment}
\usepackage{amscd}

\advance \topmargin by -\headheight
\advance \topmargin by -\headsep
     
\evensidemargin \oddsidemargin
\newtheorem{theorem}{Theorem}[section]
\newtheorem{lemma}[theorem]{Lemma}
\newtheorem{corollary}[theorem]{Corollary}

\theoremstyle{definition}

\theoremstyle{remark}

\numberwithin{equation}{section}

\def\bfa{{\mathbf a}}

\def\bfd{{\mathbf d}} 

\def\bfF{{\mathbf F}}

\def\bft{{\mathbf t}}
\def\bfu{{\mathbf u}}

\def\bfx{{\mathbf x}}
\def\bfy{{\mathbf y}}
\def\bfz{{\mathbf z}}

\def\calA{{\mathcal A}}  

\def\calB{{\mathcal B}} 
\def\calC{{\mathcal C}}

\def\calO{{\mathcal O}}

\def\calT{{\mathcal T}}
\def\calU{{\mathcal U}}

\def\A{{\mathbb A}}
\def\AA{{\mathbf A}}
\def\C{{\mathbb C}}
\def\G{{\mathbb G}}
\def\N{{\mathbb N}}
\def\R{{\mathbb R}}
\def\Z{{\mathbb Z}}\def\Q{{\mathbb Q}}

\def\gra{{\mathfrak a}}

\def\grm{{\mathfrak m}}\def\grM{{\mathfrak M}}
\def\grn{{\mathfrak n}}\def\grS{{\mathfrak S}}

\def\grp{{\mathfrak p}}

\def\alp{{\alpha}} \def\bfalp{{\boldsymbol \alpha}}
\def\bet{{\beta}}  \def\bfbet{{\boldsymbol \beta}}
\def\gam{{\gamma}}
\def\bfgam{{\boldsymbol \gamma}}
\def\Gam{{\Gamma}}
\def\del{{\delta}} \def\Del{{\Delta}}
\def\iot{{\iota}}
\def\zet{{\zeta}}  

\def\tet{{\theta}}  
\def\kap{{\kappa}}
\def\lam{{\lambda}}  
\def\bflam{{\boldsymbol \lambda}}

\def\sig{{\sigma}}

\def\ome{{\omega}} 
\def\d{{\partial}}
\def\eps{\varepsilon}

\def\d{{\,{\rm d}}}

\specialcomment{comblue}{\begin{comblue}}{\end{comblue}}
\excludecomment{comblue}

\specialcomment{comgreen}{\begin{comgreen}}{\end{comgreen}}
\excludecomment{comgreen}

\def\ord{{\rm ord}}

\def\Tr{{\rm Tr}}
\def\Nm{{\rm Nm}}

\def\det{{\rm det}}

\def\dim{{\rm dim}}
\def\div{{\rm div}}

\def\Br{{\rm Br}}

\def\sm{{\rm sm}}
\newcommand{\xnu}{x^{(\nu)}}
\newcommand{\bfxnu}{\bfx^{(\nu)}}
\newcommand{\ynu}{y^{(\nu)}}
\newcommand{\bfynu}{\bfy^{(\nu)}}

\def\Xbar{{\bar{X}}}
\def\Ubar{{\bar{U}}}
\def\Gal{{\rm Gal}}
\def\Pic{{\rm Pic}}
\def\Div{{\rm Div}}

\begin{document}

\title{Norms as products of linear polynomials}
\author[Damaris Schindler and Alexei Skorobogatov]{Damaris Schindler 
and Alexei Skorobogatov}
\address{Hausdorff Center for Mathematics, Endenicher Allee 62, 53115 Bonn, Germany}
\email{damaris.schindler@hcm.uni-bonn.de}
\address{Department of Mathematics, South Kensington Campus,
Imperial College London, SW7 2BZ, United Kingdom}
\email{a.skorobogatov@imperial.ac.uk}


\begin{abstract}
Let $F$ be a number field, and let $F\subset K$ be a field extension
of degree $n$. Suppose that we are given $2r$ sufficiently general 
linear polynomials in $r$ variables over $F$.
Let $X$ be the variety over $F$ such that the $F$-points of $X$
bijectively correspond to the representations of 
the product of these polynomials by a norm from $K$ to $F$.
Combining the circle method with descent we
prove that the Brauer--Manin obstruction is the only obstruction 
to the Hasse principle and weak approximation on 
any smooth and projective model of $X$.
\end{abstract}
\maketitle

\section{Introduction} 

Let $K/F$ be an extension of number fields of degree $n\geq 2$. We fix 
a basis $\xi_1,\ldots, \xi_n$ of $K$ as an $F$-vector space, and write
$N(\bfz)$ for the norm form $N_{K/F}(z_1\xi_1+\ldots +z_n \xi_n)$,
where $\bfz = (z_1,\ldots, z_n)$.
Let $L_1(\bft),\ldots,L_{2r}(\bft)$, where
$\bft=(t_1,\ldots, t_r)$, be non-zero linear functions 
with coefficients in $F$, not necessarily homogeneous. 
Consider the Diophantine equation
\begin{equation}\label{eqn1.0}
\prod_{i=1}^{2r}L_i^{e_i}(\bft)= c N(\bfz),
\end{equation}
where $c \in F^*$, and $e_1,\ldots, e_{2r}$ are positive
integers. It is known that already for $r=1$ and $e_1=e_2=1$
weak approximation for (\ref{eqn1.0}) can fail.
Thus one is naturally led to investigate whether the Brauer--Manin obstruction
controls the Hasse principle and weak approximation
on smooth and projective varieties birationally equivalent to the affine
hypersurface (\ref{eqn1.0}).
For $r=1$ this was proved for $F=\Q$ in \cite{HBSko} and \cite{CHS} 
(see also \cite{CT}),
and recently generalised to an arbitrary number field $F$ in \cite{SJ11}.
In this paper, 
which is independent of \cite{SJ11}, we combine 
the circle method of Hardy and Littlewood with the method of descent
of Colliot-Th\'el\`ene and Sansuc to extend these results to 
$r\geq 1$ and any number field $F$.

For the circle method part we require the functions $L_i$ 
to be sufficiently general. More precisely, we assume the following condition.

\smallskip

Condition I. {\em Let ${\mathcal L}$ be the set of linear functions
$\{1,L_1,\ldots, L_{2r}\}$. For each $L\in{\mathcal L}$
there exist subsets ${\mathcal A}\subset {\mathcal L}$ and 
${\mathcal B}\subset {\mathcal L}$ of linearly independent functions
such that $|{\mathcal A}|=|{\mathcal B}|=r+1$ and 
${\mathcal A}\cap{\mathcal B}=\{L\}$.}

\smallskip

%This obviously holds if no $r+1$ of 
%the functions $1, L_1,\ldots, L_{2r}$ are linearly dependent.

Our main result is the following theorem. 

\begin{theorem}\label{thm1.0}
Let $F$ be a number field. If $L_1,\ldots, L_{2r}$
satisfy Condition I, then the Brauer--Manin obstruction is the only 
obstruction to the Hasse principle and weak approximation on any smooth 
and proper model of the affine hypersurface $X$ given by 
$(\ref{eqn1.0})$. 
When the set of $F$-points of $X$ is not empty, it is Zariski dense in $X$.
\end{theorem}

The calculation of the Brauer group of a smooth and proper model of $X$ is
a non-trivial open problem, see \cite{CHS}, \cite{W1}, \cite{W2} for
some results in this direction. The following corollary to Theorem \ref{thm1.0}
is based on the simplest case when the Brauer group is trivial, 
pointed out in \cite[Cor. 2.7]{CHS}.

\begin{corollary}\label{cor1}
In the assumptions of Theorem \ref{thm1.0} assume further that either

{\rm (i)} $(e_1,\ldots,e_{2r})=1$ and $K$ does not contain a cyclic extension of $F$
of degree $d$ such that $1<d<n$, or

{\rm (ii)} $n$ is prime and $K$ is not a Galois extension of $F$.

\noindent Let $X_\sm$ be the smooth locus of $X$. Then 
the image of the natural map
$$ X_\sm(F)\rightarrow \prod_\nu X_\sm(F_\nu),$$
where $F_\nu$ ranges over all completions of $F$,
is dense in the product of local topologies. 
\end{corollary}

Our descent argument is summarised in Theorem \ref{new}
which closely follows \cite{CHS}.
We construct a smooth partial compactification $X'$ of
$X_\sm$ such that $X'$ has no non-constant
invertible regular 
functions and the geometric Picard group of $X'$ is torsion-free. We define
a convenient class of $X'$-torsors, called `vertical' torsors. Such
$X'$-torsors always exist and are birationally equivalent to the product
of the variety $Y$ given by 
\begin{equation}\label{eqn1.1}
\sum_{j=1}^{2r} a_{ij} N(\bfz_j)+a_{i,2r+1}=0,\quad 1\leq i\leq r,
\end{equation}
where $a_{ij}\in F$ and $\bfz_j=(z_{n(j-1)+1},\ldots,z_{nj})$,
and the affine variety $N(\bfz)=a$, for some $a\in F^*$. 

We always write $s=2r+1$ and $m=[F:\Q]$. 
It is easy to show (see the proof of 
Theorem \ref{thm1.0} in Section \ref{desc}) 
that Condition I implies that the
coefficient $(r\times s)$-matrix $A=(a_{ij})$ satisfies the
following rank condition.

\smallskip

Condition II. {\em If we remove any column of $A$, the remaining
columns can be partitioned into two $(r\times r)$-matrices of full rank.}

\smallskip

Using descent we deduce 
Theorem \ref{thm1.0} from Theorem \ref{thm1.1} below and the 
well known theorem of Sansuc that 
the Brauer--Manin obstruction is the only obstruction to 
the Hasse principle and weak approximation on smooth compactifications
of principal homogeneous spaces of tori.

\begin{theorem}\label{thm1.1}
Let $Y$ be the affine variety given by {\rm (\ref{eqn1.1})} where
the matrix $A$ satisfies Condition II. Then 
the image of the natural map
\begin{equation*}
Y_\sm(F)\rightarrow \prod_\nu Y_\sm(F_\nu),
\end{equation*}
where $F_\nu$ ranges over all completions of $F$,
is dense in the product of local topologies.  
\end{theorem}

This theorem establishes the Hasse principle and
weak approximation for $Y_\sm$.
To prove it we homogenise the system of equations
(\ref{eqn1.1}) using an extra norm form, and then apply
the Hardy--Littlewood circle method over $F$.
Write $\calO_F$ for the ring of integers of $F$.
Let $B=(b_{ij})$ be an $(r\times s)$-matrix with 
entries in $\calO_F$ that satisfies Condition II. 
Write $\bfx=(\bfx_1,\ldots,\bfx_s)$, and set
$$f_i(\bfx)=\sum_{j=1}^s b_{ij}N(\bfx_j)$$
for $1\leq i\leq r$.
We look for integer solutions of the system of equations 
\begin{equation}\label{eqn1.2}
f_i(\bfx)=0,\quad  1\leq i\leq r,
\end{equation} 
in a certain box. Moreover, we want these solutions to
satisfy congruence conditions. 
%We use the same convention of notation for the vector $\bfd$ as for $\bfx$.
Let $\grn\subset \calO_F$ be an integral ideal, and let 
$\ome_1,\ldots,\ome_m$ be a $\Z$-basis of $\grn$. 
This is also a basis of the real vector space $V=F\otimes_\Q \R$. 
Fix $\kappa>0$ and $\bfu=(u_1,\ldots,u_{ns})\in V^{ns}$, and define the box
$\calB=\calB(\bfu,\kappa)\subset V^{ns}$ as 
\begin{equation*}
\calB(\bfu,\kappa)=\{\bfx\in V^{ns}: |x_{ij}-u_{ij}|\leq \kappa 
\mbox{ for } 1\leq i\leq ns \mbox{ and } 1\leq j\leq m \},
\end{equation*}
where the real variables $x_{ij}$ are defined by $x_i=\sum_{j=1}^m x_{ij}\ome_j$,
and similarly for $u_{ij}$. Fix also a vector $\bfd\in (\calO_F)^{ns}$. 
We are interested in the number of solutions
\begin{equation*}
N(\calB,P)=|\{\bfx\in (P\calB)\cap \grn^{ns}: f_i(\bfx +\bfd)=0 \mbox{
  for } 1\leq i\leq r\}|,
\end{equation*}
where $P$ is large. Theorem
\ref{thm1.1} is a corollary of the following result.
 
\begin{theorem}\label{thm1.2}
Let $\{\xi_1,\ldots,\xi_n\}\subset {\mathcal O}_K$ be a basis of $K$ 
as an $F$-vector space, and let $B$ be a matrix that satisfies 
Condition II. If
\begin{equation*}
{\rm rk} \left( \frac{\partial f_i}{\partial x_j}(\bfx)\right)=r
\end{equation*}
for any $\bfx\in\calB$, then
\begin{equation*}
N(\calB,P)=\mu(\calB)P^{mn(r+1)}+O(P^{mn(r+1)-\eta})
\end{equation*}
for some $\eta >0$,
where $\mu(\calB)$ is the product of local densities given
explicitly in equation (\ref{eqn7.1}) below. Moreover,
if the system of equations $f_i(\bfx+\bfd)=0$ has a
nonsingular solution in $\grn_\nu^{ns}$ for all finite places $\nu$
of $F$, and the system of equations $f_i(\bfx)=0$ has a nonsingular 
solution in $\calB$, then $\mu(\calB)>0$. 
\end{theorem}
Here $\grn_\nu=\grn{\mathcal O}_\nu$,
where ${\mathcal O}_\nu$ is the ring of integers of $F_\nu$.
We specify the condition on the box $\calB$ to simplify 
the treatment of the singular integral.

Theorem \ref{thm1.2} is of interest because,
on the one hand, the number of variables in (\ref{eqn1.2}) is 
linear in the number of the equations and their degrees.
On the other hand, the catalogue of examples in which 
the circle method has been 
applied to number fields with conclusions independent of the
degree of the field, is extremely small (see for example
\cite{Bir1961b} and \cite{Ski1997}). 
Our approach relies on the work of
Birch, Davenport and Lewis \cite{BDL1962} and of Heath-Brown and
one of the authors \cite{HBSko}. For our
system of linear equations with variables replaced by norm forms
we obtain an asymptotic formula without weights, in contrast to \cite{SJ11}.

The paper is organised as follows. In Section \ref{desc} 
we describe vertical torsors for
the variety over a field $F$ whose $F$-points bijectively 
correspond to the representations of the values of an arbitrary 
polynomial in several variables by a norm from 
a finite extension $K/F$. We apply descent 
and deduce our main Theorem \ref{thm1.0} from Theorem 
\ref{thm1.1}, and prove Corollary \ref{cor1}. 
In Section 3 we set up the circle
method over number fields and prove Theorems \ref{thm1.1} and 
\ref{thm1.2}. 
 
\smallskip

\textbf{Acknowledgements.} The first author was partly 
supported by a DAAD scholarship. The second author was supported 
by the Centre Interfacultaire Bernoulli of the Ecole Polytechnique
F\'ed\'erale de Lausanne.
We are grateful to Prof. T.D. Wooley for suggesting this problem to us.

\section{Descent} \label{desc}

We begin by proving a slightly more general
descent statement than the one needed to deduce Theorem \ref{thm1.0}
from Theorem \ref{thm1.1}.

Let $F$ be a field of characteristic zero with an algebraic 
closure $\bar F$ and the Galois group $\Gam_F= \Gal (\bar F / F)$. 
When $X$ is an $F$-variety we write $\Xbar = X\times_F \bar F$.
We denote the smooth locus of $X$ by $X_\sm$.

Let $N(\bfz)$ be a norm form attached to a field extension
$K/F$ of degree $n$. Define
a hypersurface $X\subset \A_F^{r+n}$ by the equation
$P(\bft)=N(\bfz)$, where $P(\bft)$ is a non-constant polynomial in
$F[\bft]=F[t_1,\ldots,t_r]$. The closed subset $Y\subset \A_F^r$
given by $P(\bft)=0$ is the union of irreducible components
$Y=Y_1\cup\ldots\cup Y_d$. For each $i=1,\ldots,d$ 
choose a geometrically irreducible
component $Y'_i\subset\bar Y_i$, and let $F_i\subset \bar F$ 
be the invariant subfield of the stabiliser of $Y'_i$ in
$\Gam_F$. Let $P_i(\bft)\in F_i[\bft]$ be an absolutely irreducible 
polynomial in $\bft=(t_1,\ldots,t_r)$ such that $Y'_i$ is given by
$P_i(\bft)=0$. 
Let us use $N_{F_i/F}$ as an abbreviation for the norm
$N_{F_i(\bft)/F(\bft)}$.
Then $N_{F_i/F}(P_i(\bft))$ is an irreducible polynomial 
in $F[\bft]$ such that $Y_i$ is given by $N_{F_i/F}(P_i(\bft))=0$.
Thus the hypersurface $X\subset \A_F^{r+n}$ can be given by
\begin{equation}\label{uno}
\prod_{i=1}^{d}N_{F_i/F}(P_i(\bft))^{e_i}= c N(\bfz),
\end{equation}
where $c\in F^*$ and $e_1,\ldots, e_d$ are positive integers. 

Let $\bfy_i$ be a variable with values in $K\otimes_F F_i$,
for $i=1,\ldots,d$. Consider the quasi-affine subvariety 
$V\subset \A^r_F\times\prod_{i=1}^d 
R_{K\otimes_F F_i/F}(\A^1)$ defined by
\begin{equation}
P_i(\bft)=\varrho_i N_{K\otimes_F F_i/F_i}(\bfy_i)\not=0,
\label{tres}
\end{equation}
where $\varrho_i\in F_i^*$ and $i=1,\ldots,d$. 

\begin{theorem} \label{new}
Let $F$ be a number field. Suppose that for any
$\varrho_i\in F_i^*$, $i=1,\ldots,d$, the variety $V$
satisfies the Hasse principle and weak approximation. Then
the Brauer--Manin obstruction is the only 
obstruction to the Hasse principle and weak approximation on any smooth 
and proper model of the affine hypersurface $X$.
\end{theorem}

\begin{proof}
Let $\pi: X\to \A_F^r$ be the
morphism defined by the projection to coordinates $t_1,\ldots, t_r$. 
Define $U_0\subset \A_F^r$ as the open subset given by $N(\bfz)\neq 0$,
and $U=\pi^{-1}(U_0)$.
We note that $\bar U\cong \bar U_0\times \G_{m,\bar F}^{n-1}$, and this
implies $\Pic(\bar U)=0$. 
%This also shows that $\bar F[U]^*$ is generated by
%the classes of the functions $L_i$ and $u_j$.

We can write $N(\bfz)$ as the product $\prod_{i=1}^n u_i(\bfz)$ of 
linearly independent linear forms with coefficients in
$\bar F$. It is easy to check that the complement to the union
of closed subsets given by $u_i=u_j=0$ for all $i\neq j$, is smooth. 
Thus $\pi(X_\sm)=\A_F^r$ and $U\subset X_\sm$. 

Recall that $R_{K/F}(\G_{m,K})$ is a torus over $F$ defined as the Weil 
restriction of the multiplicative group $\G_{m,K}$. The module of characters
of $R_{K/F}(\G_{m,K})$ is the induced $\Gam_F$-module 
$\Z[\Gam_F/\Gam_K]$ that will be denoted by $\Z[K/F]$. The norm torus
$T$ is the kernel of the surjective homomorphism
$R_{K/F}(\G_{m,K})\to \G_{m,F}$ given by the norm $N_{K/F}$,
so $T$ is the affine hyperplane $N(\bfz)=1$. The module of
characters $\widehat T$ fits into the exact sequence of $\Gam_F$-modules
$$0\to\Z\to \Z[K/F]\to\widehat T\to 0,$$ 
where $1\in\Z$ goes to the sum of canonical generators of $\Z[K/F]$.

It is known (see, e.g., \cite{CHS2}) that $T$, 
like any other torus, has a smooth equivariant compactification. 
This is a smooth, projective and geometrically integral variety $T^c$ 
over $F$ with an action of 
$T$ that contains an open $T$-orbit isomorphic to $T$. 
The contracted product $U^c=U\times^T T^c$ can be defined 
as the quotient of $U\times T^c$ by the simultaneous action 
of $T$ on both factors. Thus the morphism $\pi:U\to U_0$ extends to a smooth
and proper morphism $U^c\to U_0$, and $U$ is open and dense in $U^c$. Moreover,
each geometric fibre of $U^c\to U_0$ is a smooth compactification
of $T$. Let $X'$ be the scheme over $F$ obtained
by gluing $X_\sm$ and $U^c$ along $U$. The argument in \cite[p. 71]{CHS}
shows that $X'$ is separated, hence $X'$ is a variety. We denote the natural
morphism $X'\to\A^r_F$ also by $\pi$. Since the generic fibre $X'_\eta$ 
of this morphism
is projective and geometrically integral, by restricting an invertible 
regular function $f$ on $\bar X'$ to $X'_\eta$ we see that $f\in \bar F(\A^r_F)$.
However, the morphism $\pi:X'\to\A^r_F$ is surjective, hence if the 
divisor of $f$
in $\A^r_F$ is non-zero, the divisor of $f$ in $\bar X'$ is non-zero too.
We conclude that $\bar F[X']^*=\bar F^*$, that is, $\bar X'$ has 
no non-constant invertible regular functions.

It is clear that the geometrically irreducible components
of the hypersurface $Y_i\subset\A_F^r$ 
form a $\Gam_F$-stable $\Z$-basis of the free abelian group
$\Z[F_i/F]$. We thus have a natural isomorphism of $\Gam_F$-modules
$$\bar F[U_0]^*/\bar F^*=\oplus_{i=1}^d \Z[F_i/F].$$
The geometrically irreducible components of $X'\setminus U^c$ 
form a $\Gam_F$-stable $\Z$-basis of the free abelian group
of divisors on $\Xbar'$ with support outside of $\bar U^c$:
$$\Div_{\Xbar'\setminus \bar U^c}(\Xbar')=
\Z[K/F]\otimes (\oplus_{i=1}^d \Z[F_i/F]).$$
We call an irreducible divisor $D\subset \bar X'$ horizontal if $\pi$ induces
a dominant map $D\to \A_{\bar F}^r$. The subgroup of 
$\Div_{\Xbar'\setminus \Ubar}(\Xbar')$ generated by horizontal divisors is
$\Div_{\Ubar^c\setminus \Ubar}(\Ubar^c)$, which is isomorphic to 
$\Div_{\bar T^c\setminus\bar T}(\bar T^c)$
as a $\Gam_F$-module, see \cite[Lemma 2.1]{CHS}.
We obtain a direct sum decomposition of $\Gam_F$-modules
$$\Div_{\bar X'\setminus\bar U}(\bar X')=
\Div_{\bar T^c\setminus\bar T}(\bar T^c) \oplus 
\Div_{\Xbar'\setminus \bar U^c}(\Xbar').$$
There is a commutative diagram of $\Gam_F$-modules with exact rows and columns
\begin{equation}\begin{array}{ccccccccc}
&&0&&0&&0&&\\
&&\downarrow&&\downarrow&&\downarrow&&\\
0&\to&\bar F[U_0]^*/\bar F^*&\to& 
\Div_{\Xbar'\setminus \bar U^c}(\Xbar') &\to& 
\widehat T\otimes (\oplus_{i=1}^d \Z[F_i/F]) &\to&0\\
&&\downarrow&&\downarrow&&\downarrow&&\\
0&\to&\bar F[U]^*/\bar F^*&\to&\Div_{\bar X'\setminus\bar U}(\bar X')&\to&
\Pic(\bar X')&\to&0\\
&&\downarrow&&\downarrow&&\downarrow&&\\             
0&\to&\widehat T&\to&\Div_{\bar T^c\setminus\bar T}(\bar T^c)
&\to&\Pic(\bar T^c)&\to&0\\
&&\downarrow&&\downarrow&&\downarrow&&\\   
&&0&&0&&0&& 
\end{array}\label{due}
\end{equation}
constructed in the same way as the diagram in
\cite[Prop. 2.2]{CHS}.
The injective maps in the top and middle rows are induced by the map 
$\div_{\bar X'}$ sending a function 
to its divisor in $\bar X'$. The middle row is exact because
$\bar F[X']^*=\bar F^*$ and $\Pic(\bar U)=0$.
The vertical maps from the middle row to the bottom row are given by
the restriction to the generic fibre of $\pi:\bar X'\to\A^r_{\bar F}$.
We refer to \cite[Prop. 2.2]{CHS} for the identification 
of the modules and the maps in the bottom row.
The smooth and projective variety $\bar T^c$ is rational, hence
$\Pic(\bar T^c)$ is torsion-free. From the exactness of the right hand column
of (\ref{due}) we see that $\Pic(\bar X')$ is torsion-free.

We refer to \cite[Section 2]{Sko} for more details on torsors, in 
particular, for the definition of the type of a torsor under a torus.
Let $\lambda$ be the injective map of $\Gam_F$-modules
from the right hand column of (\ref{due}). 
We shall call a torsor $\calT\to X'$ 
of type $\lambda$ a \emph{vertical} torsor. Let
$\calT_U$ be the restriction of $\calT$ to $U\subset X'$.

\begin{lemma}\label{3.1}
Vertical $X'$-torsors exist. For each such torsor $\calT$ there exist
a principal homogeneous space $E$ of the torus $T$,
and $\varrho_i\in F_i^*$, $i=1,\ldots,d$, such that
$\calT_U=E\times V$, where $V$ is defined in {\rm (\ref{tres})}.
\end{lemma}

\begin{proof}
Recall that $\Pic(\bar U)=0$, and take
the two upper rows of our diagram (\ref{due}) as the 
diagram (4.21) of \cite{Sko}. 
An immediate application of the local description of torsors 
\cite[Thm. 4.3.1]{Sko} shows that 
$\calT_U$ is given by (\ref{uno}) together with (\ref{tres}). Let
$E$ be the principal homogeneous space of $T$ with the equation
$$\prod_{i=1}^{d}N_{F_i/F}(\varrho_i)^{e_i}=cN(\bfz).$$
Multiplying $\bfz$ by $\prod_{i=1}^{d}N_{K\otimes_F F_i/K}(\bfy_i)^{e_i}$
we get an isomorphism $\calT_U=E\times V$.
\end{proof}

We resume the proof of Theorem \ref{new}.

Recall that $\Br_0(X)$ is the image of the natural
map $\Br(F)\to\Br(X)$, and $\Br_1(X)$ is the kernel of the natural
map $\Br(X)\to \Br(\bar X)$. 

It suffices to consider one smooth and proper model of $X$ over $F$.
We can take it to be a smooth and projective variety $X^c$ 
that contains $X'$ as a dense open subset. 
(Since $F$ has characteristic zero,
such a variety exists by Hironaka's theorem.) 

Let $\AA$ be the ring of ad\`eles of the number field $F$.
Let $(M_\nu)\in X^c(\AA)^{\Br_1(X^c)}$ be a collection of local points
$M_\nu\in X^c(F_\nu)$, one for each place
$\nu$ of $F$, orthogonal to $\Br_1(X^c)$. 
By a theorem of Grothendieck,
$\Br_1(X^c)$ is naturally a subgroup of $\Br_1(X')$.
We have seen that $\bar F[X']^*=\bar F^*$ 
and $\Pic(\bar X')$ is torsion-free.
It is well known that this implies that
$\Br_1(X')/\Br_0(X')$ is a subgroup of $H^1(F,\Pic(\bar X'))$,
and hence is finite. Thus we can use
\cite[Prop.~1.1]{CTSk} (a consequence of Harari's `formal lemma') which
says that the natural injective map of topological spaces
$$X'(\AA)^{\Br_1(X')}\rightarrow X^c(\AA)^{\Br_1(X^c)}
=\big(\prod_\nu X^c(F_\nu)\big)^{\Br_1(X^c)}$$
has a dense image. Thus we can assume $(M_\nu)\in X'(\AA)^{\Br_1(X')}$.
Furthermore, using the finiteness of $\Br_1(X')/\Br_0(X')$
and the fact that the value of an element of $\Br_1(X')$ at a point of
$X'(F_\nu)$ is locally constant in the topology of $F_\nu$, we can assume
without loss of generality that $M_\nu\in U(F_\nu)$ for all $\nu$.

The main theorem of the descent theory
of Colliot-Th\'el\`ene and Sansuc states that every point in
$X'(\AA)^{\Br_1(X')}$ is in the image of the map
$\calT_0(\AA)\to X'(\AA)$, where $\calT_0\to X'$ is a universal torsor 
(see \cite[Section 3]{CS87} and \cite[Thm.~6.1.2(a)]{Sko}).
Thus we can find a point $(N_\nu)\in\calT_0(\AA)$ such that
the image of $N_\nu$ in $X'$ is $M_\nu$ for all $\nu$. 

The structure group of $\calT_0\to X'$ is the N\'eron--Severi torus $T_0$
defined by the property $\widehat T_0=\Pic(\bar X')$.
The right hand column of (\ref{due}) gives rise to 
the dual exact sequence of tori
$$1\to T_1\to T_0\to T_2\to 1,$$ which is the definition of $T_1$ and $T_2$.
The quotient $\calT=\calT_0/T_1$ is an $X'$-torsor with the structure 
group $T_2$. The type of $\calT\to X'$ is the natural map
$$\widehat T_2=\widehat T\otimes (\oplus_{i=1}^d \Z[F_i/F])
\longrightarrow\Pic(\bar X'),$$
so $\calT$ is a vertical torsor. Since $\calT_0$ is a universal torsor,
we have $\bar F[\calT_0]^*=\bar F^*$ and $\Pic(\bar \calT_0)=0$,
hence $\Br_1(\calT_0)=\Br_0(\calT_0)$. 

Let $(P_\nu)\in \calT(\AA)$ be the image of $(N_\nu)$. By the functoriality
of the Brauer--Manin pairing we see that 
$(P_\nu)\in \calT(\AA)^{\Br_1(\calT)}$. By Lemma \ref{3.1}
the restriction of $\calT$ to $U$ is isomorphic to $E\times V$,
where $V$ is given by (\ref{tres}). Let $S$ be a finite set of places
of $F$ containing all the places where we need to approximate.
By assumption we can find
an $F$-point in $V$ close to the image of $P_\nu$ in $V(F_\nu)$
for $\nu\in S$.

The argument in \cite[p. 85]{CHS} shows that there is
an $F$-point in $E$ close to the image of $P_\nu$ in $E(F_\nu)$
for $\nu\in S$.
We reproduce this
argument for the convenience of the reader.
Let $E^c$ be a smooth compactification of $E$. The projection
$\calT_U=E\times V\to E$ extends to a rational map $f$
from the smooth variety $\calT$ to the projective variety $E^c$.
By a standard result of algebraic geometry there is an
open subset $W\subset\calT$ with complement $\calT\setminus W$
of codimension at least 2 in $\calT$ such that $f$ is a 
morphism $W\to E^c$. By Grothendieck's purity theorem the natural
restriction maps $\Br(\calT)\to\Br(W)$ and
$\Br(\bar\calT)\to\Br(\bar W)$ are isomorphisms.
Hence $\Br_1(\calT)\to\Br_1(W)$ is also an isomorphism.
Thus $f^*\Br_1(E^c)\subset\Br_1(W)$ is contained in
$\Br_1(\calT)$, and so the image of $(P_\nu)$ in $E^c$
belongs to $E^c(\AA)^{\Br_1(E^c)}$. By Sansuc's theorem,
$E(F)$ is a dense subset of $E^c(\AA)^{\Br_1(E^c)}$.

We conclude that there is a point in $\calT(F)$ which is arbitrarily
close to $P_\nu$ for $\nu\in S$. The image of this point in $X^c(F)$ approximates
$(M_\nu)$. This finishes the proof of Theorem \ref{new}.
\end{proof}

\begin{proof}[Proof of Theorem \ref{thm1.0}]
%The equation (\ref{eqn1.0}) is 
Consider the particular case of 
(\ref{uno}) where for each $i=1,\ldots,d$ we have $F_i=F$ 
and $P_i(\bft)$ is a linear polynomial $L_i(\bft)$.
Then the natural projection $\A^r_F\times\prod_{i=1}^d R_{K/F}(\A^1_K)
\to \prod_{i=1}^d R_{K/F}(\A^1_K)$ defines 
an isomorphism $V=V_0\times \A^g_F$, where $g$ is the dimension
of the kernel of the linear map $F^r\to F^d$ given by the homogeneous
parts of $L_1,\ldots,L_d$, and $V_0$ is defined as follows.
For some
$\varrho_1,\ldots, \varrho_d\in F^*$ the variety $V_0$ is given by
the equations in $K$-variables $\bfy_1,\ldots, \bfy_{d}$:
\begin{equation}\label{deqn1}
\sum_{i=1}^d\lam_i\varrho_i N(\bfy_i)+\lam_{d+1}=0, \quad \quad
N(\bfy_i)\not=0, \ \ i=1,\ldots,d,
\end{equation}
for all vectors $(\lam_1,\ldots, \lam_{d+1})\in F^{d+1}$ satisfying
\begin{equation}\label{eqnlam}
\sum_{i=1}^{d}\lam_i L_i(\bft)+\lam_{d+1}=0
\end{equation}
identically in $\bft=(t_1,\ldots, t_r)$.

In the situation of Theorem \ref{thm1.0} we have $d=2r$,
and the linear polynomials $L_1,\ldots, L_{2r}, L_{2r+1}=1$ satisfy 
Condition I. In particular, their linear span has dimension $r+1$
and we have $V=V_0$. Thus the vectors
$\bflam = (\lam_1,\ldots, \lam_{2r+1})$ satisfying equation (\ref{eqnlam})
form an $r$-dimensional subspace $\Lambda\subset F^{2r+1}$. Let
$\bflam^{(1)},\ldots, \bflam^{(r)}$ be a basis of $\Lambda$. 
Set $a_{ij}=\lam_j^{(i)}\rho_j$ for $1\leq i\leq r$ and $1\leq j\leq 2r+1$,
where we set $\rho_{2r+1}=1$.
Thus $V$ is a dense open subset of the affine variety given
by the system of equations (\ref{eqn1.1}).

An easy linear algebra argument shows that if $L_1,\ldots, L_{2r+1}$ 
satisfy Condition I, then the matrix $(a_{ij})$ satisfies Condition II. 
Indeed, take any $j_0$ from $1$ to $2r+1$ and write 
$\{L_1,\ldots, L_{2r+1}\}$ as the union
of subsets of linearly independent functions $\calA = \{L_{j_0},
L_{j_1},\ldots, L_{j_r}\}$ and $\calB = \{ L_{j_0},L_{j_{r+1}},\ldots,
L_{j_{2r}}\}$. Since the elements of $\calB$ are linearly independent,
there exists a unique vector in $\Lambda$ whose coordinates with subscripts
$j_1,\ldots,j_r$ are arbitrary elements of $F$. It follows that the matrix
$(a_{ij_l})_{1\leq i\leq r , 1\leq l\leq r}$ has full rank.
The elements of $\calA$ are also linearly independent, so
the matrix $(a_{ij_{r+l}})_{1\leq i\leq r , 1\leq l\leq r}$
has full rank too. Hence the matrix $(a_{ij})$ satisfies Condition II,
and now the result follows from Theorems \ref{new} and \ref{thm1.1}.
\end{proof}

\begin{proof}[Proof of Corollary \ref{cor1}]
The arguments in the proof of \cite[Cor. 2.7]{CHS} apply \emph{verbatim}
in our situation, establishing $\Br(X'_c)=\Br_0(X'_c)$ in case (i). 
The proof of the statement in the example \cite[p. 77-78]{CHS} gives 
the same conclusion in case (ii).
\end{proof}

\section{Circle method}

\subsection{Preliminaries}
We write $\Tr=\Tr_{F/\Q}$ for the trace from $F$ to $\Q$. Let 
\begin{equation*}
\calC = \{x\in F: \Tr (xy) \in \Z \mbox{ for all } y\in \calO_F\}
\end{equation*}
be the inverse different. We extend $\Tr$ to 
a linear form $V=F\otimes_\Q\R\to \R$. 
Choose a $\Z$-basis $\zet_1,\ldots,\zet_m$
of $\calO_F$. Let $\rho_1,\ldots,\rho_m$ be the dual basis of $\calC$
defined by the property that the matrix with entries
$\Tr (\zet_i\rho_j)$ is the identity matrix. Then any $x\in F$
can be written as 
\begin{equation}
x=\sum_{i=1}^m \Tr(x\rho_i)\zeta_i. \label{dva}
\end{equation}
We set $$f_{ik}(\bfx)=\Tr (\rho_k f_i(\bfx)).$$

Let $(c_{ik})_{1\leq i\leq r,1\leq k\leq q}$ be
a $(r\times q)$-matrix with entries in the $\R$-algebra $V$.
We shall say that the rank of this matrix is $r$ 
if it defines a surjective linear map $V^q\to V^r$. 
Recall that 
$\ome_1,\ldots,\ome_m$ is a $\Z$-basis of the ideal $\grn\subset \calO_F$.
We attach to $(c_{ik})$ the $(rm\times qm)$-matrix with real entries
$(\Tr (c_{ik}\rho_j\ome_l))_{(i,j),(k,l)}$, where we use
the lexicographic ordering of the pairs 
$(i,j)$, $1\leq i\leq r$, $1\leq j\leq m$, and the pairs $(k,l)$, 
$1\leq k\leq q$, $1\leq l\leq m$. 

The following observation will be often used in this paper.

\begin{lemma}\label{linalg}
The matrix $(c_{ik})$ with entries in $V$ has rank $r$ if and only if 
the matrix $(\Tr (c_{ik}\rho_j\ome_l ))_{(i,j),(k,l)}$ with entries in 
$\R$ has rank $mr$.
\end{lemma}

\begin{proof}
Take any $d_1,\ldots,d_r\in V$ and 
write $d_i=d_{i1}\ome_1+\ldots+d_{im}\ome_m$. 
If there exist $\mu_1,\ldots,\mu_q\in V$  such that
$ c_{i1}\mu_1+\ldots+  c_{iq}\mu_q=d_i$, for $1\leq i\leq r$,
we write $\mu_k=\mu_{k1}\ome_1+\ldots +\mu_{km}\ome_m$
and then obtain
\begin{equation}
\sum_{k=1}^q \sum_{l=1}^m \mu_{kl}\Tr (c_{ik}\rho_j\ome_l) = 
\sum_{p=1}^m d_{ip}\Tr (\ome_p\rho_j),\label{odin}
\end{equation}
for $1\leq i\leq r$ and $1\leq j\leq m$. 
The $(m\times m)$-matrix $\Tr (\ome_p\rho_j)$
is invertible by the non-degeneracy of the bilinear form
$\Tr(xy): F\times F\to\Q$. Thus the rank of
$(\Tr( c_{ik}\rho_j\ome_l))$ is $mr$. Conversely, from (\ref{odin}) 
using (\ref{dva}) we deduce  
$$\sum_{k=1}^q c_{ik}\sum_{l=1}^m \mu_{kl}\ome_l = d_i.$$
This finishes the proof of the lemma.
\end{proof}

Let $Z$ be the affine variety over $F$ defined by the system of 
equations (\ref{eqn1.2}). The Weil restriction $R_{F/\Q}(Z)$ 
is the variety over $\Q$ defined by the system of equations 
$f_{ij}(\bfx)=0$ for $1\leq i\leq r$ and $1\leq j\leq m$. 
For any $\Q$-algebra $S$ there is a natural bijection of points 
$R_{F/\Q}(Z)(S)= Z(S\otimes_\Q F)$. 
A useful consequence of Lemma \ref{linalg} is the observation that a 
$V$-point of $Z$ is singular if and only if the corresponding 
$\R$-point on the variety $R_{F/\Q}(Z)$ is singular.
%A useful consequence is the observation that a point of the variety over $V$ given by
%$f_i(\bfx)=0$, for $1\leq i\leq r$, is singular if and only if
%the same point is a singular point of the variety
%over $\R$ given by the polynomials $f_{ij}(\bfx)=0$, for $1\leq i\leq r$
%and $1\leq j\leq m$. 
Indeed, by the chain rule we have
\begin{equation}\label{eqn5.0}
\frac{\partial f_{ij}}{\partial x_{kl}}(\bfx)= \Tr \left( \rho_j 
\frac{\partial}{\partial x_{kl}} f_i(\bfx)\right)=\Tr \left( 
\rho_j\ome_l\frac{\partial f_i}{\partial x_k}(\bfx)\right),
\end{equation}    
and the statement follows from Lemma \ref{linalg} with $q=ns$ and 
$c_{ik}=\partial f_i/\partial x_k$.

\subsection{Exponential sums} \label{es}
For $1\leq j\leq s$ we define $\calB_j$ to be the set
\begin{equation*}
\calB_j=\{ \bfx_j\in V^n: |u_{ik}-x_{ik}|\leq \kappa \mbox{ for } 
n(j-1)+1\leq i\leq nj \mbox{ and } 1\leq k\leq m\}.
\end{equation*}
Then we have
\begin{equation*}
\calB = \calB_1\times \ldots\times \calB_s.
\end{equation*}
Now we introduce the exponential sums
\begin{equation*}
S_j(\bfbet)=\sum_{\bfx_j\in(P\calB_j)\cap \grn^n}e(\Tr(\bfbet N(\bfx_j +\bfd_j))),
\quad\quad 1\leq j\leq s,
\end{equation*}
where we write $\bfbet=\bet_1\rho_1+\ldots+\bet_m\rho_m$, and identify 
$\bfbet$ with the vector $(\bet_1,\ldots,\bet_m)\in \R^m$. 
Consider the linear forms 
\begin{equation*}
\bflam_j = \sum_{i=1}^r b_{ij}\bfalp_i,\quad\quad 1\leq j\leq s.
\end{equation*}
For $\bflam_j$ and $\bfalp_i$ we use the same conventions as for 
$\bfbet$. By orthogonality we have
\begin{equation}\label{eqn3.1}
N(\calB,P)=\int_{[0,1]^{mr}} S_1(\bflam_1)\ldots S_s(\bflam_s)\d\bfalp,
\end{equation}
where we write $\bfalp=(\bfalp_1,\ldots,\bfalp_r)$ and 
$\d\bfalp=\d\alp_{11}\ldots\d\alp_{rm}$.\par

Next we turn towards a form of Weyl's inequality for the exponential 
sums $S_j(\bfbet)$, which we deduce from Birch's work \cite{Bir1961}. 

\begin{lemma}\label{lem3.1}
Let $\Del$ and $\tet$ be positive integers satisfying
$2^{n-1}\Del<\tet$. Let $j$ be an integer such that $1\leq j\leq s$.
Then 

{\rm (i)} either we have $|S_j(\bfbet)|\ll P^{mn-\Del}$, or

{\rm (ii)} there is an integer $q$ such that $1\leq q\leq P^{m(n-1)\tet}$,
and integers $a_1,\ldots,a_m$ such that $\gcd (a_1,\ldots,a_m,q)=1$ and 
\begin{equation*}
2|q\bet_i - a_i|\leq P^{-n+m(n-1)\tet}, \quad\quad 1\leq i\leq m.
\end{equation*}
\end{lemma}

\begin{proof}
Since $j$ is fixed, we drop it from notation. 
Recall that $\bfz= (z_1,\ldots, z_n)$, where $z_k\in V$, so 
we can write $z_k= z_{k1}\ome_1+\ldots +z_{km}\ome_m$.
Let $Y\subset \C^{mn}$ be the Zariski closed subset given by 
\begin{equation}\label{eqn3.2}
{\rm rk}\left(\frac{\partial\Tr(\rho_i N(\bfz))}{\partial z_{kl}}\right)_{i,(k,l)}<m,
\end{equation}
where $1\leq i\leq m$, and the pairs $(k,l)$, 
where $1\leq k\leq n$ and $1\leq l\leq m$, are ordered
lexicographically as in the previous section. 
By Lemma \ref{linalg} and (\ref{eqn5.0}) this is equivalent to 
\begin{equation*}
\frac{\partial N(\bfz)}{\partial z_{k}}=0, \quad\quad 1\leq k\leq n.
\end{equation*}
However, by Euler's formula for homogeneous polynomials we have
\begin{equation*}
n N(\bfz)=\sum_{k=1}^n z_{k} \frac{\partial N(\bfz)}{\partial z_{k}}.
\end{equation*}
Since $N(\bfz)$ is not identically zero, we see that $\dim(Y)\leq mn-1$.\par
We have
\begin{equation*}
S_1(\bfbet)=\sum_{\bfz\in (P\calB_1)\cap\grn^n}e\left(\sum_{i=1}^m 
\bet_i \Tr( \rho_i N(\bfz+\bfd))\right).
\end{equation*}
Applying \cite[Lemmas 3.2 and 3.3]{Bir1961} to the system of equations
$\Tr(\rho_iN(\bfz+\bfd))=0$ in $mn$ variables $z_{kl}$
we obtain that either one of the alternatives of our lemma holds, 
or we have
\begin{equation*}
(n-2)mn+ \dim(Y)\geq (n-1)mn -2^{n-1}\Del/\tet -\eps,
\end{equation*}
for some $\eps > 0$. This is impossible since $2^{n-1}\Del<\tet$.
\end{proof}

In the following we always choose $\Del >0$ small enough so that 
the condition of Lemma \ref{lem3.1} is satisfied.

\subsection{The circle method}
We start this section with choosing appropriate major and minor arcs. For
a positive real number 
$\tet$, integers $q$ and $a_{11},\ldots,a_{rm}$ define the 
major arc
$\grM_{\bfa,q}(\tet)$ to be the set of $\bfalp\in [0,1]^{mr}$ such that 
\begin{equation*}
|q\alp_{ij}-a_{ij}|\leq q P^{-n+mr(n-1)\tet}, \quad\quad 1\leq i\leq r, \quad
1\leq j\leq m.
\end{equation*}
Then $\grM (\tet)$ is the union of the major arcs
\begin{equation*}
\grM (\tet)= \bigcup_{1\leq q\leq P^{mr(n-1)\tet}}\bigcup_{\bfa} \grM_{\bfa,q}(\tet),
\end{equation*}
where the second union is over all vectors $\bfa$ satisfying 
$\gcd (a_{11},\ldots,a_{rm},q)=1$ and $0\leq a_{ij}< q$. 
We choose $\tet$ sufficiently small such that all the major arcs in 
the union of $\grM (\tet)$ are disjoint, which is possible by 
\cite[Lemma 4.1]{Bir1961}. As usual, we define the minor arcs 
$\grm (\tet)$ as the complement 
$\grm (\tet)= [0,1]^{mr}\setminus \grM(\tet)$ to the major arcs.\par

Let us now treat the contribution of the minor arcs to 
the integral (\ref{eqn3.1}). For this we need the following lemma, 
which appeared in a similar way in the work of Birch, Davenport and 
Lewis \cite{BDL1962}. 

\begin{lemma}\label{lem4.1}
For any $\eps >0$ we have
\begin{equation*}
\int_{[0,1]^m}|S_j(\bfbet)|^2 \d\bfbet \ll P^{mn+\eps},\quad \quad 
1\leq j\leq s.
\end{equation*}
\end{lemma}

\begin{proof}
By orthogonality we see that this integral is equal to the number of 
solutions $\bfz_1,\bfz_2\in (P\calB_j)\cap\grn^n$ of the equation
\begin{equation*}
N(\bfz_1+\bfd)=N(\bfz_2+\bfd).
\end{equation*}
Write $N_{K/\Q}$ for the norm from $K$ to $\Q$. 
For $z\in K$ we denote by $z^{(l)}$, where $1\leq l\leq mn$, the 
conjugates of $z$. By the transitivity of 
norm, the number above is bounded by the number of solutions 
$z_1,z_2\in \calO_K$ of
\begin{equation*}
N_{K/\Q}(z_1)=N_{K/\Q}(z_2),\ \ 
\max_{1\leq l\leq mn}|z_1^{(l)}|\leq C_1 P, \ \  
\max_{1\leq l\leq mn}|z_2^{(l)}|\leq C_1 P,
\end{equation*}
for some constant $C_1$. For an integer $u$ let $r(u)$ be
the number of $z\in \calO_K$ such that
\begin{equation}
N_{K/\Q}(z)=u, \quad\quad
\max_{1\leq l\leq mn}|z^{(l)}|\leq C_1P.
\label{*}
\end{equation}
Now the integral in the lemma is bounded by 
\begin{equation*}
\sum_{|u|\leq C_2P^{mn}}r(u)^2,
\end{equation*}
for some large enough $C_2$. To prove the lemma it is enough
to show that for $u$ in this sum we have $r(u)\ll P^\eps$. One can find this
result in Lemma 4.3 in \cite{Ple75}. For convenience, we repeat a proof here.
Group together the solutions $z$ of (\ref{*}) that generate the same
principal ideal $(z)$. Let us denote the norm of an integral ideal
$\gra$ by $\Nm(\gra)$. By the unique factorisation of prime ideals
the number of integral ideals $\gra$ of norm $\Nm(\gra)=u>0$ is
bounded by some constant times $u^\eps$. 
Now we fix a solution $z$ of (\ref{*}), if it exists, and
consider the number $A(u,z)$ of solutions $\tilde{z}$ of (\ref{*})
such that $z$ and $\tilde{z}$ differ by a unit. Let
$\iot_1,\ldots,\iot_T$ be fundamental units of $K$. For some integers
$v_1,\ldots,v_T$ we have
\begin{equation}\label{eqn4.1}
\tilde{z}=\zet \iot_1^{v_1}\ldots\iot_T^{v_T}z,
\end{equation}
where $\zet$ is a root of unity. Furthermore, we have
\begin{equation*}
\sum_{l=1}^{mn}\log |\tilde{z}^{(l)}|=\log |u|\ll \log P,
\end{equation*}
and $\log |\tilde{z}^{(l)}|\leq \log (C_1P)$ for all $l$. Therefore, 
there is a
constant $C_3$ such that for large values of $P$ we have the bound
\begin{equation*}
\big|\log |\tilde{z}^{(l)}|\big|\leq C_3\log P, \quad\quad 1\leq l\leq mn.
\end{equation*}
The same estimate is true for $|\log|z^{(l)}||$. 
Using (\ref{eqn4.1}) we see that
\begin{equation*}
\left|\sum_{i=1}^Tv_i\log |\iot_i^{(l)}|\right|\ll \log P, 
\quad\quad 1\leq l\leq mn.
\end{equation*}
By Dirichlet's unit theorem the rank of
the matrix $(\iot_i^{(l)})$, where $1\leq i\leq T$ and $1\leq l\leq mn$,
is $T$, and hence $A(u,z)\ll (\log P)^T$. 
This implies $r(u)\ll P^\eps$. 
\end{proof}

\begin{lemma}\label{lem4.2}
There exists $\eta >0$ such that we have 
\begin{equation*}
\int_{\grm(\tet)}|S_1(\bflam_1)\ldots S_s(\bflam_s)|\d\bfalp = O(P^{mn(r+1)-\eta}).
\end{equation*}
\end{lemma}

\begin{proof}
In the first part of the proof we show that if $\bfalp$ is of minor arc 
type, then so is one of the $\bflam_i$ for some possibly different 
parameter $\tet'$. For this let $\grm_j(\tet)$ be the set
of $\bfalp\in \grm(\tet)$ such that $|S_j(\bflam_j)|\ll P^{mn-\Delta}$. Assume
that $\bfalp\notin \cup_{j=1}^s \grm_j(\tet)$ and choose $\tet'<\tet$ 
such that we still have $2^{n-1}\Delta<\tet'$. Then we apply
Lemma \ref{lem3.1} and find integers $1\leq q_j \leq P^{m(n-1)\tet'}$ and
$a_{jl}$ for $1\leq j\leq s$ and $1\leq l\leq m$ with the property that 
\begin{equation*}
2|q_j\lam_{jl}-a_{jl}|\leq P^{-n+m(n-1)\tet'},
\end{equation*}
for all $j$ and $l$. For simplicity of notation we assume next that the matrix
$(b_{ij})_{1\leq i,j\leq r}$ has full rank, which is possible after renaming
indices since the matrix $B$ has full rank by assumption. Thus, there are
$c_{ij}\in k$ such that 
\begin{equation*}
\bfalp_i=\sum_{j=1}^rc_{ij}\bflam_j,
\end{equation*}
for $1\leq i\leq r$. Next we define 
$\widetilde{\bflam}_j=q_j^{-1}(a_{j1}\rho_1+\ldots+a_{jm}\rho_m)$,
and
\begin{equation*}
\tilde{a}_{ik}=\Tr \left(\zet_k\sum_{j=1}^rc_{ij}\widetilde{\bflam}_j\right),
\quad\quad 1\leq i\leq r, \quad 1\leq k\leq m.
\end{equation*}
By construction there is an integer
$q\ll P^{mr(n-1)\tet'}$ such that $q \tilde{a}_{ik}\in\Z$ for all $i$ and
$k$. We can estimate
$$
|\alp_{ik}-\tilde{a}_{ik}|=|\Tr \big(\zet_k\sum
c_{ij}(\bflam_j-\widetilde{\bflam}_j)\big)|
\ll \max_{j,l}\big(|q_j^{-1}a_{jl}-\lam_{jl}|\big)
\ll P^{-n+m(n-1)\tet'}.
$$
It follows that $\bfalp\in\grM (\tet)$, and hence
$\grm(\tet)=\cup_{j}\grm_j(\tet)$.\par
We estimate the contribution from the sets
$\grm_j(\tet)$ to the integral in the lemma separately. 
For simplicity of notation we assume that 
$|S_s(\bflam_s)|\ll P^{mn-\Delta}$ and that both the
first $r$ columns and the next $r$ columns of the matrix $B$ form submatrices
of full rank, which we can do without loss of generality 
by Condition II. Using the Cauchy--Schwarz inequality we estimate
\begin{align*}
\int_{\grm_s(\tet)}|S_1(\bflam_1)\ldots S_s(\bflam_s)|\d\bfalp&\ll
P^{mn-\Delta}\int_{\grm_s(\tet)}|S_1(\bflam_1)\ldots
S_{2r}(\bflam_{2r})|\d\bfalp\\
&\ll P^{mn-\Delta} I_1^{1/2}I_2^{1/2},
\end{align*}
where
\begin{equation*}
I_1=\int_{[0,1]^{mr}}|S_1(\bflam_1)\ldots S_r(\bflam_r)|^2\d\bfalp,
\end{equation*}
and $I_2$ of analogous form. We now perform a change of variables in the
integral $I_1$. For this note that
$$
\lam_{jl}=\Tr \left(\zet_l\sum_{i=1}^r b_{ij}\bfalp_i\right)
=\sum_{i=1}^r\sum_{k=1}^m \alp_{ik}\Tr (b_{ij}\rho_k\zet_l).$$
Order the pairs $(i,k)$ and $(j,l)$ lexicographically, and let 
$M$ be the $(mr\times mr)$-matrix with entries 
$\Tr (b_{ij}\rho_k\zet_l)$. Then $M$ has full rank and
$\Tr(b_{ij}\rho_k\zet_l)\in\Z$ for all $i,j,k$ and $l$. Write
$\d\bflam$ for the Lebesgue measure $\d\lam_{11}\ldots\d\lam_{rm}$. 
By 1-periodicity of our exponential sums we have
$$
I_1=\frac{1}{\det M}\int_{M[0,1]^{mr}}|S_1(\bflam_1)\ldots S_r(\bflam_r)|^2\d\bflam
\ll\int_{[0,1]^{mr}}|S_1(\bflam_1)\ldots S_r(\bflam_r)|^2\d\bflam.
$$
The integral $I_2$ can be treated in the very same way as $I_1$.
Our lemma now follows from Lemma \ref{lem4.1}.
\end{proof}

We analyse the major arcs following Birch's approach in \cite{Bir1961}.
Let us write
\begin{equation*}
S(\bfalp)=S_1(\bflam_1)\ldots S_s(\bflam_s)=
\sum_{\bfx\in (P\calB)\cap \grn^{ns}} e(\alp_{11} f_{11}(\bfx+\bfd)+\ldots
+\alp_{rm} f_{rm}(\bfx+\bfd)),
\end{equation*}
where the exponential sums $S_i(\bflam_i)$ were defined in
the beginning of Section \ref{es}. 
We define the exponential sums
\begin{equation*}
S_{\bfa,q}=\sum_{\bfx\in (\Z/q)^{mns}} e\left(\sum_{i=1}^r\sum_{j=1}^m a_{ij}f_{ij}(\bfx+\bfd)/q\right).
\end{equation*}
For $\bfgam\in \R^{mr}$ we define
\begin{equation*}
I(\bfgam)= \int_{\bft \in \calB}e\left(\sum_{i=1}^r\sum_{j=1}^m \gam_{ij}
  f_{ij}(\bft)\right) \d\bft, 
\quad \quad
J(P)=\int_{|\bfgam|\leq P}I(\bfgam)\d\bfgam.
\end{equation*}
In this last integral we use the notation
$|\bfgam|=\max_{ij}|\gam_{ij}|$. For the vector $\bft$ we use the 
same conventions as were adopted in the introduction for the vector $\bfx$.

\begin{lemma}\label{lem4.3}
For a small enough $\tet>0$ there exists $\eta > 0$ such that we have
\begin{equation*}
\int_{\grM(\tet)}S(\bfalp)\d\bfalp = \grS (P) J(P^{mr(n-1)\tet}) P^{mn(r+1)}+
O(P^{mn(r+1)-\eta}),
\end{equation*}
where
\begin{equation*}
\grS (P)= \sum_{q\leq P^{mr(n-1)\tet}}q^{-mns}\sum_\bfa S_{\bfa,q},
\end{equation*}
and $\bfa$ ranges over all vectors with $0\leq a_{ij} < q$ and 
$\gcd (a_{11},\ldots,a_{rm},q)=1$.
\end{lemma}

\begin{proof}
This is a combination of Lemmas 5.1 and 5.5 of \cite{Bir1961}
together with the argument of \cite[Section 9]{Schmidt1985} which
ensures that the error introduced by replacing 
$f_{ij}(\bfx+\bfd)$ by $f_{ij}(\bfx)$ is small enough.
\end{proof}

\subsection{Singular Integral}

By assumption the system of equations $f_i(\bft)=0$ has no 
singularities in the box $\calB$. By Lemma \ref{linalg} and the remark 
following it, the corresponding system $f_{ij}=0$ is also 
non-singular on $\calB$. Splitting the box into smaller ones
if necessary we may assume that the same $(mr\times mr)$-minor of the Jacobian
matrix of the $f_{ij}$ has full rank on the whole box. For simplicity of
notation we assume furthermore that it is the minor $C$ given by
$\frac{\partial f_{ij}}{\partial t_{nk,l}}$ for $1\leq i,k\leq r$ and $1\leq
j,l\leq m$. Here again we order the pairs $(i,j)$ and $(k,l)$ 
lexicographically. After splitting the box $\calB$ into even smaller boxes,
so that on each of them the inverse function theorem becomes
applicable, we can perform a coordinate transformation in the integral
$I(\bfgam)$ introduced in the last section, as follows.
Set $u_{ij}=f_{ij}(\bft)$ for $1\leq i\leq r$ and $1\leq j\leq m$, 
and write $\bfu$ for the vector $(u_{11},\ldots,u_{rm})$. 
After renaming the indices of the vector $\bft$ we
can write $\bft = (\bft',\bft'')$ with $\bft' \in \R^{m(ns-r)}$ and 
$\bft''\in \R^{mr}$, so that $\bft''$ consists of all the coordinates 
of $\bft$ of the form $t_{nk,l}$ for $1\leq k\leq r$ and $1\leq l\leq m$.
Let $V(\bfu)$ be the set of all $\bft '\in \R^{m(ns-r)}$ such that 
there is some $\bft ''\in \R^{mr}$
with the corresponding $\bft = (\bft',\bft'') \in\calB$ and 
$u_{ij}=f_{ij}(\bft '',\bft ')$
for all $1\leq i\leq r$ and $1\leq j\leq m$. Define 
\begin{equation*}
\psi(\bfu)=\int_{\bft '\in V(\bfu)} |\det\, C (\bft)|^{-1} \d\bft ',
\end{equation*}
where $\bft$ is implicitly given by $\bfu$ and $\bft '$. Then we obtain 
\begin{equation*}
I(\bfgam)=\int_{\R^{rm}} \psi(\bfu)e(\bfgam \cdot \bfu)\d\bfu,
\end{equation*}
where we write $\bfgam\cdot\bfu$ for the scalar product
$\sum_{i=1}^r\sum_{j=1}^m \gam_{ij}u_{ij}$. 

Our next goal is to show,
using the Fourier inversion theorem, that $J(P)$ absolutely 
converges to $\psi(0)$ when $P\to \infty$. 
First we need a lemma.

\begin{lemma}\label{lem5.1b}
Let $\calA$ be a rectangular box in $\R^D$.
For $1\leq i\leq m$ let $F_i(\bfz)\in \R[\bfz]$ be polynomials with real coefficients
in $\bfz = (z_1,\ldots,z_D)$. Let $l$ be an integer such that $0\leq l\leq D-m$.
Assume that all $(m\times m)$-minors of the matrix
\begin{equation*}
\left( \frac{\partial F_i}{\partial z_j}\right)_{1\leq i\leq m,\ 1\leq j\leq m+l}
\end{equation*}
have full rank on some open subset $\calU\supset \calA$. 
Let $G:\calU\to \R$ be a smooth function. Then for any $\bet_1,\ldots,\bet_m\in\R$ one has
\begin{equation*}
\left| \int_{\calA} G(\bfz) e(\bet_1 F_1(\bfz)+\ldots + \bet_m F_m(\bfz)) \d z_1\ldots \d z_D\right| \ll (\max_i|\bet_i|)^{-l-1},
\end{equation*}
where the implied constant depends only on $\calA$ and the functions $F_i$ and $G$. 
\end{lemma}

\begin{proof}
Write $\bfbet \bfF (\bfz)= \bet_1 F_1(\bfz)+\ldots +\bet_m F_m(\bfz)$.
Consider the differential $D$-form on $\calU$:
\begin{equation*}
\ome = G(\bfz) e(\bfbet \bfF(\bfz))\d z_1\wedge \ldots \wedge \d z_D.
\end{equation*}
For any smooth functions $\phi_i(\bfz)$ on $\calU$ we define the $(D-1)$-form
\begin{equation*}
\mu = \sum_{i=1}^m G(\bfz)\phi_i(\bfz) e(\bfbet \bfF(\bfz)) \d z_1 \wedge \ldots \wedge \widehat{\d z_i} \wedge \ldots \wedge \d z_D,
\end{equation*}
where $\widehat{\d z_i}$ means that $\d z_i$ is omitted. Then $\d\mu = \ome_1+\ome_2$, where
\begin{equation*}
\ome_1 = \sum_{i=1}^m (-1)^{i+1} \frac{\partial}{\partial z_i}(G(\bfz)\phi_i(\bfz)) e(\bfbet \bfF(\bfz)) \d z_1\wedge \ldots \wedge\d  z_D,
\end{equation*}
and
\begin{equation*}
\ome_2=\sum_{i=1}^m (-1)^{i+1} G(\bfz)\phi_i(\bfz)\frac{\partial}{\partial z_i} e(\bfbet \bfF(\bfz)) \d z_1\wedge \ldots \wedge \d z_D.
\end{equation*}
Without loss of generality we assume $|\bet_1|=\max_i|\bet_i|$.
We claim that the functions $\phi_i(\bfz)$ can be chosen so that $\ome_2=\bet_1\ome$
for all $\bet_1,\ldots,\bet_m$. For this we have to solve 
\begin{equation*}
2\pi \sqrt{-1}\sum_{i=1}^m (-1)^{i+1} \phi_i(\bfz) \left( 
\bet_1 \frac{\partial F_1}{\partial z_i} (\bfz)+\ldots +
\bet_m \frac{\partial F_m}{\partial z_i} (\bfz)\right)= \bet_1,
\end{equation*}
where $\bfz\in \calU$. The $(m\times m)$-matrix 
$$J=\left( \frac{\partial F_i}{\partial z_j}\right)_{1\leq i,j\leq m}$$
is invertible by assumption, hence
we can choose the functions $\phi_i(\bfz)$ to be the functions defined by
the following equality of row vectors:
$$2\pi \sqrt{-1}\big((-1)^{i+1} \phi_i(\bfz)\big)=(1,0,\ldots,0)J^{-1}.$$
Now we have $\d\mu = \ome_1+\bet_1 \ome$ on $\calU$, 
and the Stokes theorem gives
\begin{equation}\label{eqnlem1}
\int_{\calA} \ome = \frac{1}{\bet_1} \left( \int_{\partial \calA} \mu - \int_{\calA}\ome_1\right),
\end{equation}
where $\partial \calA$ is the boundary of $\calA$. 
The integrals in the right hand side of (\ref{eqnlem1}) have the same form as the integral we started with. 
Thus, we can iterate the above procedure $l$ times for each occuring term. 
In the end we estimate each integral by its $L^1$-bound using the trivial estimate $|e(\bfbet \bfF(\bfz))|\leq 1$. 
This produces the desired inequality.
\end{proof}

Now we can prove that the integral $J(P)$ is absolutely convergent. 
Define $ \gam_i=\gam_{i1}\rho_1+\ldots +\gam_{im}\rho_m$. From the definition of $I(\bfgam)$ we have
$$
I(\bfgam) =\int_{\bft\in\calB} e\left( \sum_{i=1}^r \Tr \big(\gam_i f_i(\bft)\big)\right) \d\bft.$$
This can be rewritten as
$$\int_{\bft\in\calB} e\left( \sum_{i=1}^r \Tr \big(\gam_i \sum_{j=1}^s b_{ij}N(\bft_j)\big)\right) \d\bft
= \prod_{j=1}^s \nu_j(\bfgam),
$$
where
\begin{equation*}
\nu_j(\bfgam) = \int_{\bft_j\in\calB_j} e\left(\Tr \big(\sum_{i=1}^r\gam_i b_{ij} N(\bft_j)\big)\right) \d\bft_j.
\end{equation*}
In Theorem \ref{thm1.2} we have assumed that
\begin{equation*}
{\rm rk} \left( \frac{\partial f_i}{\partial x_k}\right)=r,
\end{equation*}
on the box $\calB$. Without loss of generality we assume that the matrix 
$(b_{ij})_{1\leq i,j\leq r}$ has full rank, and after possibly dissecting the box $\calB$ into smaller ones, we assume that 
\begin{equation*}
\frac{\partial N(\bft_j)}{\partial t_{n(j-1)+1}}\neq 0
\end{equation*}
on $\calB_j$, for all $1\leq j\leq r$.\par
Next we note that
$$
\sum_{i=1}^r \gam_ib_{ij} = \sum_{k=1}^m \rho_k \Tr \left(\zet_k \sum_{i=1}^r \gam_i b_{ij}\right)
= \sum_{k=1}^m \rho_k \Tr\left( \sum_{i=1}^r\sum_{l=1}^m \gam_{il} b_{ij} \rho_l\zet_k\right).
$$
Since the matrix $(b_{ij})_{1\leq i,j\leq r}$ with entries in $F$ has full rank, we have
\begin{equation*}
\det \big(\Tr(b_{ij}\rho_l\zet_k)_{\substack{(1,1)\leq (i,l)\leq (r,m)\\ (1,1)\leq (j,k)\leq (r,m)}}\big)\neq 0,
\end{equation*}
which follows from Lemma \ref{linalg}. Thus we have the relation
\begin{equation*}
|\bfgam|\asymp \max_{(j,k)} \left| \sum_{i=1}^r\sum_{l=1}^m \gam_{il}\Tr (b_{ij}\rho_l\zet_k)\right|,
\end{equation*}
where the implied constants only depend on the numbers $b_{ij}$ and the bases $\zet_k$ and $\rho_l$. 
Next we choose $j_0$ where the maximum is attained, and assume $j_0=1$ for simplicity of notation.\par
Now we apply Lemma \ref{lem5.1b} to the integral $\nu_1(\bfgam)$. For this we set
\begin{equation*}
F_k(\bft_1)=\Tr (\rho_k N(\bft_1)),
\end{equation*}
for $1\leq k\leq m$. By the above assumptions and Lemma \ref{linalg} we have
\begin{equation*}
\det \left(\frac{\partial F_k(\bft_1)}{\partial t_{1l}}\right)_{1\leq k,l\leq m}\neq 0,
\end{equation*}
on the box $\calB_1$. Lemma \ref{lem5.1b} implies the bound
\begin{equation*}
\left| \int_{\calB_1} e\left(\Tr \big(( \sum_{i=1}^r \gam_i b_{ij}) N(\bft_1)\big)\right) \d\bft_1\right| \ll |\bfgam|^{-1}.
\end{equation*}
Since the matrix $B$ satisfies Condition II, we can assume
that the matrices $(b_{ij})_{\substack{1\leq i\leq r\\ 2\leq j\leq r+1}}$ and 
$(b_{ij})_{\substack{1\leq i\leq r\\ r+2\leq j\leq s}}$ have full rank, 
possibly after renaming the indices. For a large real number $T$ we obtain the estimate
\begin{align*} 
\int_{T<|\bfgam|\leq 2T} |I(\bfgam)|\d\bfgam &\ll \sup_{T<|\bfgam|\leq 2T}|\nu_1(\bfgam)| \int_{|\bfgam|\leq 2T}\prod_{j=2}^s |\nu_j(\bfgam)|\d\bfgam\\
&\ll T^{-1} J_1(2T)^{1/2}J_2(2T)^{1/2},
\end{align*}
where
\begin{equation*}
J_1(T)=\int_{|\bfgam|\leq T}\prod_{j=2}^{r+1} |\nu_j(\bfgam)|^2\d\bfgam,
\end{equation*}
and similarly for $J_2(T)$. To establish the absolute convergence of $J(P)$ it is now sufficient to show that 
$J_i(T)\ll T^\eps$ for $i=1,2$.\par
For this we consider the exponential sum
\begin{equation*}
|S_2(\bflam_2)\ldots S_{r+1}(\bflam_{r+1})|^2,
\end{equation*}
and perform the circle method analysis of the preceding 
sections with respect 
to this exponential sum instead of $S(\bfalp)$. The second part of the proof of Lemma \ref{lem4.2} gives the estimate
\begin{align*}
\int_{\grM(\tet)} |S_2(\bflam_2)\ldots S_{r+1}(\bflam_{r+1})|^2\d\bfalp &\ll \int_{[0,1]^{mr}}|S_2(\bflam_2)\ldots S_{r+1}(\bflam_{r+1})|^2\d\bfalp \\ &\ll T^{mnr+\eps}.
\end{align*}
Furthermore, we have
\begin{equation*}
\int_{\grM(\tet)}|S_2(\bflam_2)\ldots S_{r+1}(\bflam_{r+1})|^2\d\bfalp = 
\widetilde{\grS}(T) J_1(T^{mr(n-1)\tet}) T^{mnr}+O(T^{mnr-\eta})
\end{equation*}
for some $\eta >0$, where the singular series is
\begin{equation*}
\widetilde{\grS}(T)=\sum_{q\leq T^{mr(n-1)\tet}}q^{-2mnr}\sum_{\bfa}|S_{\bfa,q}^{(2)}\ldots S_{\bfa,q}^{(r+1)}|^2.
\end{equation*}
The term $q=1$ and $a_{11}=\ldots= a_{rm}=0$ produces the lower bound
$\widetilde{\grS} (T)\geq 1$. Thus we have 
$J_1(T^{mr(n-1)\tet})\ll T^\eps$,
as desired. Since the same arguments apply also to $J_2$, we see that $J(P)$ is absolutely convergent. Indeed, we have $|J(P)-\lim_{P\rightarrow \infty}J(P)|\ll P^{-1+\eps}$ for some $\eps >0$.
\begin{lemma}\label{lem5.2}
There exists $\eps >0$ such that $J(P)=\psi( 0) + O(P^{-1+\eps})$
when $P\to\infty$. Moreover, if
the system of equations $f_i(\bfx)=0$ has a nonsingular solution in 
$\calB$, simultaneously for all the infinite places of $F$, then $\psi( 0)>0$.
\end{lemma}

\begin{proof}
The first statement follows from a form of the 
Fourier inversion theorem \cite[Cor. 1.21]{SteinWeiss}, 
which can be applied since $I(\bfgam)$ is
integrable, and the continuity of $\psi(\bfu)$ which is explained, 
for example, in \cite[Section 6]{Bir1961}. 
Next, let $\tau_1,\ldots,\tau_m$ be the $m$ different embeddings 
$F\to\bar{\Q}$. Assume that we are given a solution $\bft\in\calB$ of the
system of equations $\tau_l(f_i(\bft))=0$ for all $1\leq i\leq r$ and $1\leq
l\leq m$. Then $\bft$ is a solution of $f_i(\bfx)=0$ in $V$ since
$\det(\tau_l(\ome_j))\neq 0$, and thus a non-singular solution of the system
$f_{ij}(\bfx)=0$. Therefore, $\psi(0)$ is positive as the integral of a
positive integrand over a domain of positive measure. 
\end{proof}

\subsection{Singular Series}
Our next goal is to establish the absolute convergence of the singular series
$\grS (P)$ for $P\to \infty$. This is done using a method similar to that of
\cite{HBSko}.\par

As usual, we order the pairs $(l,j)$ lexicographically.
We claim that no $(m\times m)$-minor of the matrix
\begin{equation}\label{eqn6.0}
\left( \frac{\partial}{\partial t_{lj}} \Tr \big(\rho_i
  N(\bft_1)\big)\right)_{\substack{ 1\leq i\leq m\\ (1,1)\leq (l,j)\leq (2,m)}}
\end{equation}
has determinant zero. In the opposite case
we can find integers $1\leq j_1 < \ldots < j_q \leq m$ and
$1\leq j_{q+1}< \ldots < j_m \leq m$, and rational numbers $c_1,\ldots,c_m$, not
all of them zero, such that 
$$\sum_{k=1}^q c_k \frac{\partial \Tr (\rho_i N(\bft_1))}{\partial t_{1,j_k}} +
\sum_{k=q+1}^m c_k \frac{\partial \Tr (\rho_i N(\bft_1))}{\partial t_{2,j_k}} =0$$
for all $1\leq i\leq m$. From (\ref{eqn5.0}) and the non-degeneracy of the trace we deduce
\begin{equation*}
 \sum_{k=1}^q c_k \ome_{j_k}\frac{\partial N(\bft_1)}{\partial t_1}+
\sum_{k=q+1}^m c_k \ome_{j_k} \frac{\partial N(\bft_1)}{\partial t_2} =0
\end{equation*}
identically in $\bft_1$. Now we set 
$$t_1= \sum_{k=1}^q c_k \ome_{j_k}, \quad t_2=\sum_{k=q+1}^m c_k \ome_{j_k}, \quad 
t_3=\ldots=t_n=0,$$ and obtain
\begin{equation*}
t_1 \frac{\partial}{\partial t_1} N(\bft_1)+t_2 \frac{\partial}{\partial t_2}
N(\bft_1) =0.
\end{equation*}
By Euler's identity the left hand side is equal to $m
N(t_1,t_2,0,\ldots,0)$. This is a contradiction since not all of the $c_i\in\Q$ are
zero, and thus $N(t_1,t_2,0,\ldots,0)\neq 0$. This proves the above claim, 
which we use in the proof of our next lemma.

\begin{lemma}\label{lem6.1}
The series
\begin{equation*}
\grS=\lim_{P\to\infty}\grS(P) = \sum_{q=1}^\infty q^{-mns} \sum_{\bfa} S_{\bfa,q}
\end{equation*}
is absolutely convergent, and we have $|\grS - \grS (P)|\ll P^{-\eta}$ for some $\eta >0$.
\end{lemma}

\begin{proof}
First we choose a box $\calB = \calB_1\times\ldots \times \calB_s$ in $V^{ns}$
in such a way that each $(m\times m)$-minor of the matrix (\ref{eqn6.0}) has full rank on
$\calB_1$, and similarly for all the other $\calB_i$. We choose the $\calB_i$ as
cubes of products of half open and half closed intervals. 
Stretching $\calB$ by a suitable factor we can assume that each box has side length
$1$. Note that this does not change the nonvanishing of the
$(m\times m)$-minors of (\ref{eqn6.0}). We have $S_{\bfa,q}=S(\bfalp)$,
where $P=q$ and $\alp_{ij}=a_{ij}/q$, and then we obtain
$S_{\bfa,q}=S_{\bfa,q}^{(1)}\ldots S_{\bfa,q}^{(s)}$,
where $S_{\bfa,q}^{(j)}=S_j(\bflam_j)$.
By the first part of the proof of Lemma \ref{lem4.2}, for every $\bfa$
satisfying $\gcd (q,a_{11},\ldots,a_{rm})=1$ there exists $j$ such that
\begin{equation}\label{eqn6.1}
|S_{\bfa,q}^{(j)}|\ll q^{mn-\Delta}
\end{equation}
for some $\Delta>0$. For simplicity we assume that
$j=s$, since the other contributions can be estimated in exactly the same
way. We assume as before that the submatrix of $B$ formed by the first $r$ columns,
as well as that formed by the next $r$ columns, have both full rank. We now apply the circle
method as before to the exponential sum 
\begin{equation*}
|S_1(\bflam_1)\ldots S_r(\bflam_r)|^2,
\end{equation*}
instead of $S(\bfalp)$. By the second part of the proof of Lemma
\ref{lem4.2} we have
$$
\int_{\grM (\tet)}|S_1(\bflam_1)\ldots S_r(\bflam_r)|^2\d\bfalp\ll
\int_{[0,1]^{mr}}|S_1(\bflam_1)\ldots S_r(\bflam_r)|^2\d\bfalp\\
\ll P^{mnr+\eps}.
$$
Next, the major arc analysis gives us
\begin{equation*}
\int_{\grM (\tet)}|S_1(\bflam_1)\ldots S_r(\bflam_r)|^2\d\bfalp = \grS'(P)J'(P^{mr(n-1)\tet})
P^{mnr}+O(P^{mnr-\eta}),
\end{equation*}
for some $\eta>0$. Here the singular series is 
\begin{equation*}
\grS'(P)=\sum_{q\leq
  P^{mr(n-1)\tet}}q^{-2mnr}\sum_{\bfa}|S_{\bfa,q}^{(1)}\ldots S_{\bfa,q}^{(r)}|^2,
\end{equation*}
and the singular integral is
\begin{equation*}
J'(P)=\int_{|\bfgam|\leq P} \prod_{j=1}^r |\nu_j(\bfgam)|^2 \d\bfgam.
\end{equation*}
By Lemma \ref{lem5.1b}, applied with $l=m-1$, and the choice of our boxes $\calB_j$ we have 
\begin{equation*}
|\nu_j(\bfgam)|\ll \prod_{k=1}^m (1+|\Tr (\zet_k \sum_{i=1}^r \gam_i
b_{ij})|)^{-1}
\end{equation*}
for all $j$.

This gives the estimate
\begin{equation*}
J'(P)\ll \int_{|\bfgam|\leq P} \prod_{j=1}^r\prod_{k=1}^m (1+|\Tr (\zet_k \sum_{i=1}^r \gam_i
b_{ij})|)^{-2}\d\bfgam.
\end{equation*}
Next we use the coordinate transformation $\bfgam'= M\bfgam$ with the matrix
$M=(\Tr (b_{ij}\rho_l\zet_k ))_{(i,l),(j,k)}$, and obtain
\begin{equation*}
J'(P) \ll \int_{|\bfgam'| \leq C P} \prod_{j=1}^r\prod_{k=1}^m
(1+|\gam_{jk}'|)^{-2}\d\bfgam'
\end{equation*}
for some constant $C$. Note that $M$ has full rank by Lemma \ref{linalg} since $(b_{ij})_{1\leq
  i,j\leq r}$ is assumed to have full rank. The above equation shows that
$J'(P)$ is absolutely convergent, and thus the same arguments as in Lemma
\ref{lem5.2} imply $J'(P)=c_0+o(1)$. Here $c_0>0$ since the diagonal
solutions ensure the existence of non-singular solutions in Lemma
\ref{lem5.2}. We deduce that
\begin{equation}\label{eqn6.2}
\grS'(P)\ll P^\eps.
\end{equation}
We come back to our main argument and consider
\begin{equation*}
\grS_R=\sum_{R<q\leq 2R}q^{-mns}\sum_\bfa |S_{\bfa,q}|.
\end{equation*}
Using equation (\ref{eqn6.1}) and the Cauchy--Schwarz inequality we get
\begin{equation*}
\grS_R\ll R^{-\Delta}(\grS^{(1)}_R)^{1/2}(\grS^{(2)}_R)^{1/2},
\end{equation*}
where 
\begin{equation*}
\grS_R^{(1)}=\sum_{R<q\leq 2R}q^{-2mnr}\sum_{\bfa}|S_{\bfa,q}^{(1)}\ldots
S_{\bfa,q}^{(r)}|^2,
\end{equation*}
and $\grS_R^{(2)}$ of analogous form. Thus, equation (\ref{eqn6.2}) gives us
$\grS_R\ll R^{-\Delta +\eps}$,
which proves the lemma for $\eps$ small enough.

\end{proof}

As usual, the singular series factorises as $\grS =\prod_p c_p$,
where the product is taken over all rational primes $p$, and 
the local factors are
\begin{equation*}
c_p=\sum_{l=1}^\infty p^{-lmns}\sum_\bfa S_{\bfa,p^l}.
\end{equation*}   
Note that our polynomials $f_{ij}(\bfx+\bfd)$ have coefficients in $\Z$
since all the entries of the matrix $B$ are in $\calO_F$ and 
also we assumed that
$\{\xi_1,\ldots,\xi_n\}\subset \calO_K$ in Theorem \ref{thm1.2}. 
By standard arguments we see that the constants $c_p$ can be written as 
\begin{equation*}
c_p = \lim_{l\to \infty}p^{-l(mns-mr)}C(p,l).
\end{equation*}
Here $C(p,l)$ is the number of solutions to the simultaneous congruences
\begin{equation}\label{eqn6.3}
f_{ij}(\bfx+\bfd)\equiv 0 \bmod p^l,
\end{equation}
for $1\leq i\leq r$ and $1\leq j\leq m$, where all the components $x_{kl}$ of
$\bfx$ run through a complete set of residues modulo $p^l$.\par
Next we factorise the local densities $c_p$ further to obtain an 
interpretation in terms of the number field $F$ and 
the original system of equations $f_i(\bfx+\bfd)=0$. 
For this let $\grp_1,\ldots,\grp_t$ be the primes of $F$ which lie above $p$ and 
let $(p)=\prod_{i=1}^t\grp_i^{e_i}$ be the prime ideal factorization of the principal ideal $(p)$. 
Note that for any $z\in F$ we have $z=\sum_{j=1}^m \zet_j \Tr (\rho_j z)$, and $z\in (p^l)$ 
if and only if $\Tr(\rho_j z) \equiv 0\bmod p^l$ for $1\leq j\leq m$. 
Therefore we see that for fixed $i$ and $1\leq j\leq m$ the system of equations (\ref{eqn6.3}) is equivalent to 
\begin{equation}\label{eqn6.4}
f_i(\bfx+\bfd)\equiv 0 \bmod (p^l).
\end{equation}
Note that for some fixed $i$ a full set of residues 
$x_{i1},\ldots,x_{im}$ 
in $(\Z/p^l)^m$ corresponds to a full set of residues $x_i\in\grn$ modulo
the ideal $(p^l)\grn$ under the identification
$x_i=x_{i1}\ome_1+\ldots+x_{im}\ome_m$. 
Therefore $C(p,l)$ is equal to the
number of solutions of the system of congruences (\ref{eqn6.4}) for $1\leq
i\leq r$, where $x_j\in \grn$ run through a complete set of residues modulo
$(p^l)\grn$ for $1\leq j\leq ns$. Next write $\grn = \grn'
\prod_{i=1}^t\grp_i^{n_i}$ with $n_i\in\N$ such that $\grn'$ is coprime to
$(p)$. By a slightly modified Chinese remainder theorem we have an isomorphism
\begin{equation}\label{eqn6.5}
\grn/(p^l)\grn\to \oplus_{i=1}^t\grp_i^{n_i}/\grp_i^{le_i+n_i}.
\end{equation}
Hence $C(p,l)=\prod_{k=1}^tD(\grp_k,le_k)$,
where $D(\grp_k,l)$ is the number of solutions of the system 
$f_i(\bfx+\bfd)\equiv 0\bmod\grp_k^l$ for $1\leq i\leq r$, 
where we count solutions $x_j\in \grp_k^{n_k}$ modulo $\grp_k^{l+n_k}$
for all $1\leq j\leq ns$.\par

\begin{lemma}\label{lem6.2}
The singular series factorises as $\grS=\prod_\grp \sig_\grp$, where the
product is taken over all primes $\grp$ of $\calO_F$, and the corresponding
factors are given by
\begin{equation*}
\sig_\grp=\lim_{l\to \infty}\Nm (\grp) ^{-l(ns-r)}D(\grp,l).
\end{equation*}
Moreover, $\grS>0$ if the system of equations $f_i(\bfx+\bfd)=0$ has a
nonsingular solution in $\grn_\nu^{ns}$ for all finite places of $k$. 
\end{lemma}

\begin{proof}
By the above discussion and the multiplicativity of the norm of ideals, 
for the first part of the lemma it is enough to show that the limits in the
definition of $\sig_\grp$ exist.
For this we identify $\Nm(\grp)^{-l(ns-r)}D(\grp,l)$ with a subseries of
$\grS$. For some ideal $\gra$ let $\hat{\gra}$ be the dual given by
\begin{equation*}
\hat{\gra}=\{y\in F: \Tr (yz)\in\Z \mbox{ for all } z\in\gra\},
\end{equation*}
and note that $\hat{\gra}=\gra^{-1}\calC$. For some $z\in\calO_F$ we consider
the character $e(\Tr(yz))$ for $y\in\hat{\gra}$. This is trivial if
and only if $z\in\gra$, since $\hat{\hat{\gra}}=\gra$. Therefore we have the
orthogonality relation
\begin{equation*}
\sum_{y}e(\Tr(yz))=\left\{ \begin{array}{rc} \Nm (\gra) & \mbox{for }  z\in\gra, \\ 0
    &   \mbox{otherwise},
\end{array}\right. 
\end{equation*}
where the sum is taken over a complete set of residues $y\in \hat{\gra}$
modulo $\calC$. Note that the index of $\calC$ in $\hat{\gra}$ is just $\Nm(\gra)$. Using
this relation $r$ times we see that
\begin{equation*}
\Nm(\grp)^{-l(ns-r)}D(\grp,l)=\Nm(\grp)^{-lns}\sum_{\bfy}\sum_{\bfx}e\left(\sum_{i=1}^r
  \Tr(y_if_i(\bfx+\bfd))\right),
\end{equation*}
where the first sum is over all $\bfy\in \hat{(\grp^l)}^r$ modulo $\calC$, and the second sum
is over all $\bfx$ with $x_j\in\grp^{n_\grp}$ modulo $\grp^{l+n_\grp}$. We
write here $n_\grp$ for the power to which $\grp$ occurs in $\grn$ as we did
in the analysis preceding this lemma. Putting $y_i=y_{i1}\rho_1+\ldots+ y_{im}\rho_m$ with 
$y_{ij}=a_{ij}/q$ for some integers $a_{ij}$ and $q$, and 
using equation (\ref{eqn6.5}) for extending the summation over $\bfx$ to several sets of representatives, 
we can identify $\Nm(\grp)^{-l(ns-r)}D(\grp,l)$ with a subseries of $\grS$ as claimed above.\par
We turn to the second part of the lemma. Since $\grS$ is absolutely
convergent, it is enough to show that $\sig_\grp$ is positive if 
the system of
equations $f_i(\bfx+\bfd)=0$ has a nonsingular solution in 
$\grn_\grp^{ns}$ for a fixed
prime $\grp$. Suppose that $\bfy\in (\grn (\calO_F)_\grp)^{ns}$ is such a
nonsingular solution, and assume for simplicity of notation that the leading
minor of the corresponding Jacobian matrix has full rank. Set
\begin{equation*}
\del =\nu_\grp\left(\det \left(\frac{\partial f_i}{\partial x_j}(\bfy+\bfd)\right)_{1\leq i,j\leq  r}\right),
\end{equation*}
where we write $\nu_\grp$ for the $\grp$-adic valuation. Now set $u=
2\del+n_\grp+1$, and choose
$x_{r+1},\ldots,x_{ns}\in\calO_F$ with 
\begin{equation}\label{eqn6.6}
x_i\equiv y_i \bmod\grp^{u}.
\end{equation}
Then we have 
\begin{equation*}
f_i(y_1+d_1,\ldots,y_r+d_r,x_{r+1}+d_{r+1},\ldots,x_{ns}+d_{ns})\equiv 0 \bmod \grp^{u},
\end{equation*}
for $1\leq i\leq r$. From a slightly modified version of
\cite[Prop. 5.20]{Greenberg}, a form of Hensel's lemma, we obtain $x_1,\ldots,x_r$ with 
\begin{equation*}
f_i(\bfx+\bfd)\equiv 0 \bmod\grp^{l},
\end{equation*}
and $x_j\equiv y_j\bmod\grp^{\del+n_\grp+1}$ for $1\leq j\leq r$. If we
restrict ourselves in equation (\ref{eqn6.6}) to a complete set of residues
modulo $\grp^{l+n_\grp}$ for each $x_i$, then there are $\Nm
(\grp)^{(l-2\del-1)(ns-r)}$ choices. This shows that 
\begin{equation*}
D(\grp,l)\geq \Nm(\grp)^{ (l-2\del-1)(ns-r)},
\end{equation*}
for $l$ large enough, which proves the lemma.
\end{proof}

\subsection{Proofs of Theorems \ref{thm1.1} and \ref{thm1.2}}

\begin{proof}[Proof of Theorem \ref{thm1.2}]
We note that
\begin{equation*}
N(\calB,P)=\int_{\grM(\tet)}S(\bfalp)\d\bfalp+\int_{\grm(\tet)}
S(\bfalp)\d\bfalp.
\end{equation*}
Therefore, this theorem is a consequence of Lemma \ref{lem4.2} for the minor arc part, and
Lemma \ref{lem4.3} together with Lemmas \ref{lem5.2}, \ref{lem6.1} and \ref{lem6.2}
insofar as the main term is concerned. In particular, we get 
\begin{equation}\label{eqn7.1}
\mu(\calB)=\psi(0)\prod_\grp \sig_\grp,
\end{equation}
where the product is again taken over all primes of $k$.
\end{proof}

Now we deduce Theorem \ref{thm1.1} from Theorem \ref{thm1.2} using an argument
as in Skinner's paper (see the proof of Cor. 1 in 
Section 5 of \cite{Ski1997}).

\begin{proof}[Proof of Theorem \ref{thm1.1}]
Assume that we are given some $\eps>0$ and a finite set of places $S$ of $F$,
which we can assume to contain all infinite places. Furthermore assume that we
are given solutions 
\begin{equation*}
(\bfxnu_1,\ldots,\bfxnu_{2r})\in Y_{\rm sm} (F_\nu),
\end{equation*}
for all $\nu\in S$. We want to find $(\bfx_1,\ldots,\bfx_{2r})\in Y_{\rm sm}
(F)$ such that
\begin{equation*}
|x_j-\xnu_j|_\nu<\eps,
\end{equation*}
for all $1\leq j\leq 2nr$ and all $\nu\in S$.\par
First we write
\begin{equation*}
\xnu_{n(j-1)+1}\xi_1+\ldots+\xnu_{nj}\xi_n=\frac{\ynu_{n(j-1)+1}\xi_1+\ldots+\ynu_{nj}\xi_n}{\ynu_{2nr+1}\xi_1+\ldots+\ynu_{ns}\xi_n},
\end{equation*}
with $\ord_\nu (\ynu_j)\geq 0$ for all finite $\nu\in S$ and $1\leq
j\leq ns$. Using the Chinese
remainder theorem we can find $\bfd\in \calO_F^{ns}$ with $|d_j-\ynu_j|_\nu<\tilde{\eps}$
for all $j$ and all finite places $\nu\in S_f$. As in \cite{Ski1997}, 
if $\grp$ is the prime corresponding to a finite place $\nu$, we write
\begin{equation*}
n_\grp = \min_j\ord_\nu(d_j-\ynu_j), \quad\quad \grn=\prod_{\grp\in S_f} \grp^{n_\grp}.
\end{equation*}
\par
We turn to infinite places, and note that there is a unique $\bfu\in V^{ns}$ such that
$\tau_\nu(\bfu)= \bfynu$
for all infinite places $\nu$. Here we write $\tau_\nu$ for the embedding
corresponding to the infinite place $\nu$. Then we have
\begin{equation*}
\sum_{j=1}^{2r} a_{ij} N(\bfu_j)+a_{i,2r+1} N(\bfu_{2r+1})=0,
\end{equation*}
for $1\leq i\leq r$. We note that there is $z\in \calO_F$ such that
$z\xi_i$ is integral for all $i$. Multiplying the above equation by a
power of $z$ we can assume $\xi_1,\ldots,\xi_n$ to be integral. Next,
set $b_{ij}=\tilde{a}a_{ij}$ for $j\leq s$ so that we have
$b_{ij}\in\calO_F$ for all $i$ and $j$. By construction, $\bfu$ is then a
nonsingular solution of the system of equations $f_i(\bfx)=0$. Choosing
$\calB$ sufficiently small around $\bfu$, we can assume that for any
$\bfx\in \calB$ we have
\begin{equation*}
{\rm rk}\left( \frac{\partial f_i}{\partial x_j}(\bfx)\right) =r.
\end{equation*}
Now we can apply Theorem
\ref{thm1.2}, and get
\begin{equation*}
N(\calB,P)=\mu(\calB)P^{mn(r+1)}+o(P^{mn(r+1)}),
\end{equation*}
with some positive constant $\mu(\calB)$, since $Y_{\rm sm}(F_\nu)\neq
\emptyset$ for all places $\nu$. Let $P$ and $t$ be large integers with 
$P\equiv 1\bmod\grn^t$. For $P$ sufficiently large we then get a solution $\bfz\neq 0$ to the system of equations
$f_i(\bfz)=0$ with $\bfz-\bfd \in (P\calB)\cap\grn^{ns}$. We define $\bfx$ by
\begin{equation*}
x_{n(j-1)+1}\xi_1+\ldots+x_{nj}\xi_n=\frac{z_{n(j-1)+1}\xi_1+\ldots+z_{nj}\xi_n}{z_{2nr+1}\xi_1+\ldots+z_{ns}\xi_n},
\end{equation*}
for $1\leq j\leq 2r$. Then the $x_i$ are rational functions in $\bfz$ and hence they are
continuous in $\bfz$. For an infinite place $\nu$ we estimate
\begin{align*}
\max_j|\frac{1}{P}z_j-\ynu_j|_\nu\ll \frac{1}{P} +\kappa.
\end{align*}
For finite places $\nu\in S$ we have
$$
\max_j|\frac{1}{P} z_j-\ynu_j|_\nu\ll \max_{j} (|\frac{1}{P}z_j-z_j|_\nu+
|z_j-d_j|_\nu+|d_j-\ynu_j|_\nu)\ll \tilde{\eps},$$
for some $P\equiv 1 \bmod \grn^t$ with $t$ sufficiently large. By choosing
$\kap$ and $\tilde{\eps}$ sufficiently small and $P$ sufficiently 
large we finally obtain
\begin{equation*}
|x_j-\xnu_j|_\nu<\eps,
\end{equation*}
for all $1\leq j\leq 2nr$ and all $\nu\in S$ as required. 
\end{proof}

\bibliographystyle{amsbracket}
\providecommand{\bysame}{\leavevmode\hbox to3em{\hrulefill}\thinspace}

\end{document}